\documentclass[12pt]{amsart}
\usepackage{amsmath,amsthm,amsfonts,amssymb,mathrsfs}
\date{\today}

\usepackage{color}

\usepackage{hyperref}
 \input{xy}
 \xyoption{all}
 \xyoption{arc}

\newtheorem{theorem}{Theorem}

\newtheorem{proposition}{Proposition}
\newtheorem{corollary}{Corollary}
\newtheorem{lemma}{Lemma}
\theoremstyle{definition}
\newtheorem{example}{Example}

\newtheorem{definition}[theorem]{Definition}

\begin{document}

\title[On a semitopological extended bicyclic semigroup with adjoined zero]{On a semitopological extended bicyclic semigroup with adjoined zero}

\author[Oleg~Gutik and Kateryna Maksymyk]{Oleg~Gutik and Kateryna Maksymyk}
\address{Faculty of Mathematics, National University of Lviv,
Universytetska 1, Lviv, 79000, Ukraine}
\email{oleg.gutik@lnu.edu.ua, ovgutik@yahoo.com, kate.maksymyk15@gmail.com}

\keywords{Semigroup, bicyclic monoid, extended bicyclic semigroup, semitopological semigroup, locally compact}

\subjclass[2010]{20M10, 22A15.}

\begin{abstract}
In the paper it is shown that every Hausdorff locally compact semigroup topology on  the extended bicyclic semigroup with adjoined zero $\mathscr{C}_{\mathbb{Z}}^{\mathbf{0}}$ is discrete, but  on  $\mathscr{C}_{\mathbb{Z}}^{\mathbf{0}}$ there exist $\mathfrak{c}$ many different  Hausdorff locally compact shift-continuous topologies. Also, it is constructed on $\mathscr{C}_{\mathbb{Z}}^{\mathbf{0}}$ the unique minimal shift continuous topology  and the unique minimal  inverse semigroup topology.
\end{abstract}

\maketitle

\section{Introduction and preliminaries}\label{section-0}

We  follow the terminology of \cite{Carruth-Hildebrant-Koch-1983-1986, Clifford-Preston-1961-1967, Engelking-1989, Ruppert-1984}. In this paper all spaces are assumed to be Hausdorff. By $\mathbb{Z}$, $\mathbb{N}_0$ and $\mathbb{N}$ we denote the sets of all integers, non-negative integers and positive integers, respectively.

A \emph{semigroup} is a non-empty set with a binary associative operation. A semigroup $S$ is called \emph{inverse} if every $a$ in $S$ possesses an unique inverse, i.e. if there exists a unique element $a^{-1}$ in $S$ such that
  \begin{equation*}
    aa^{-1}a=a \qquad \mbox{and} \qquad a^{-1}aa^{-1}=a^{-1}.
\end{equation*}
A map that associates to any element of an inverse semigroup its
inverse is called the \emph{inversion}.

For a semigroup $S$, by $E(S)$ we denote the subset of all idempotents in $S$. If $E(S)$ is closed under multiplication, then we shall refer to $E(S)$ a as  the \emph{band of} $S$. The semigroup operation on $S$ determines the following partial order $\preccurlyeq$ on $E(S)$: $e\preccurlyeq f$ if and only if $ef=fe=e$. This order is called the {\em natural partial order} on $E(S)$. A \emph{semilattice} is a commutative semigroup of idempotents. A semilattice $E$ is called {\em linearly ordered} or a \emph{chain} if its natural partial order is a linear order. A \emph{maximal chain} of a semilattice $E$ is a chain which is not properly contained in any other chain of $E$.

The Axiom of Choice implies the existence of maximal chains in every partially ordered set. According to \cite[Definition~II.5.12]{Petrich-1984}, a chain $L$ is called an $\omega$-chain if $L$ is order isomorphic to $\{0,-1,-2,-3,\ldots\}$ with the usual order $\leqslant$ or equivalently, if $L$ is isomorphic to $\left(\mathbb{N}_0,\max\right)$.

The \emph{bicyclic semigroup} (or the \emph{bicyclic monoid}) ${\mathscr{C}}(p,q)$ is the semigroup with the identity $1$ generated by two elements $p$ and $q$ subject only to the condition $pq=1$. The bicyclic monoid ${\mathscr{C}}(p,q)$ is a combinatorial bisimple $F$-inverse semigroup (see \cite{Lawson-1998}) and it plays an important role in the algebraic theory of semigroups and in the theory of topological semigroups. For example, the well-known O.~Andersen's result~\cite{Andersen-1952} states that a ($0$--)simple semigroup is completely ($0$--)simple if and only if it does not contain the bicyclic semigroup. The
bicyclic semigroup cannot be embedded into the stable semigroups~\cite{Koch-Wallace-1957}.

A ({\it semi})\emph{topological semigroup} is a topological space with a (separately) continuous semigroup operation. An inverse topological semigroup with the continuous inversion is called a \emph{topological inverse semigroup}. A topology $\tau$ on a semigroup $S$ is called:
\begin{itemize}
  \item \emph{shift-continuous} if $(S,\tau)$ is a semitopological semigroup;
  \item \emph{semigroup}  if $(S,\tau)$ is a topological semigroup;
  \item \emph{inverse semigroup}  if  $(S,\tau)$ is a topological inverse semigroup.
\end{itemize}

The bicyclic semigroup admits only the discrete semigroup topology and if a topological semigroup $S$ contains it as a dense subsemigroup then ${\mathscr{C}}(p,q)$ is an open subset of $S$~\cite{Eberhart-Selden-1969}. Bertman and  West in \cite{Bertman-West-1976} extend this result for the case of Hausdorff semitopological semigroups. Stable and $\Gamma$-compact topological semigroups do not contain the bicyclic semigroup~\cite{Anderson-Hunter-Koch-1965, Hildebrant-Koch-1986}. The problem of  embedding of the bicyclic monoid into compact-like topological semigroups studied in \cite{Banakh-Dimitrova-Gutik-2009, Banakh-Dimitrova-Gutik-2010, Bardyla-Ravsky-2019??, Gutik-Repovs-2007}. Also in the paper \cite{Fihel-Gutik-2011} it was proved that the discrete topology is the unique topology on the extended bicyclic semigroup $\mathscr{C}_{\mathbb{Z}}$ such that the semigroup operation on $\mathscr{C}_{\mathbb{Z}}$ is separately continuous. An unexpected dichotomy for the bicyclic monoid with adjoined zero $\mathscr{C}^\mathbf{0}={\mathscr{C}}(p,q)\sqcup\{\mathbf{0}\}$ was proved in \cite{Gutik-2015}: every Hausdorff locally compact semitopological bicyclic semigroup with adjoined zero $\mathscr{C}^\mathbf{0}$ is either compact or discrete.

The above dichotomy was extended by Bardyla in \cite{Bardyla-2016} to locally compact $\lambda$-polycyclic semitopological monoids, and in \cite{Bardyla-2018a} to locally compact semitopological graph inverse semigroups, and also by the authors in \cite{Gutik-Maksymyk-2016} to locally compact semitopological interassociates of the bicyclic monoid with an adjoined zero, and in \cite{Gutik-2018} to locally compact semitopological $0$-bisimple inverse $\omega$-semigroups with compact maximal subgroups. The lattice of all weak shift-continuous topologies on $\mathscr{C}^0$ is described in  \cite{Bardyla-Gutik-2019}.

On the Cartesian
product $\mathscr{C}_{\mathbb{Z}}=\mathbb{Z}\times\mathbb{Z}$ we
define the semigroup operation as follows:
\begin{equation}\label{eq-0.1}
    (a,b)\cdot(c,d)=
\left\{
  \begin{array}{ll}
    (a-b+c,d), & \hbox{if }~b<c; \\
    (a,d),     & \hbox{if }~b=c; \\
    (a,d+b-c), & \hbox{if }~b>c,\\
  \end{array}
\right.
\end{equation}
for $a,b,c,d\in\mathbb{Z}$. The set $\mathscr{C}_{\mathbb{Z}}$ with
such defined operation will be called the \emph{extended bicyclic semigroup}~\cite{Warne-1968}.

In \cite{Fihel-Gutik-2011} the algebraic properties of $\mathscr{C}_{\mathbb{Z}}$ are described. It was proved there  that every non-trivial congruence $\mathfrak{C}$ on the semigroup $\mathscr{C}_{\mathbb{Z}}$ is a group congruence, and moreover the
quotient semigroup $\mathscr{C}_{\mathbb{Z}}/\mathfrak{C}$ is isomorphic to a cyclic group. It was shown that the semigroup
$\mathscr{C}_{\mathbb{Z}}$ as a Hausdorff semitopological semigroup admits only the discrete topology and the closure
$\operatorname{cl}_T\left(\mathscr{C}_{\mathbb{Z}}\right)$ of the semigroup $\mathscr{C}_{\mathbb{Z}}$ in a topological semigroup $T$ was studied there.

In  \cite{Gutik-Maksymyk-2017} we prove that the group $\mathbf{Aut}\left(\mathscr{C}_{\mathbb{Z}}\right)$ of automorphisms of the extended bicyclic semigroup $\mathscr{C}_{\mathbb{Z}}$ is isomorphic to the additive group of integers.

By $\mathscr{C}_{\mathbb{Z}}^{\mathbf{0}}$ we denote the extended bicyclic semigroup $\mathscr{C}_{\mathbb{Z}}$  with adjoined zero $\mathbf{0}$.

In this paper we show that every Hausdorff locally compact semigroup topology on the semigroup $\mathscr{C}_{\mathbb{Z}}^{\mathbf{0}}$ is discrete, but  on  $\mathscr{C}_{\mathbb{Z}}^{\mathbf{0}}$ there exist $\mathfrak{c}$ many Hausdorff locally compact shift-continuous topologies. Also, we construct on $\mathscr{C}_{\mathbb{Z}}^{\mathbf{0}}$ the unique minimal shift continuous topology  and the unique minimal inverse semigroup topology.


\section{Locally compact shift-continuous topologies on the extended bicyclic semigroup}\label{section-1}

We need the following simple statement:

\begin{proposition}[{\cite[Proposition~2.1$(viii)$]{Fihel-Gutik-2011}}]\label{proposition-2.1}
For every integer $n$ the set
\begin{equation*}
\mathscr{C}_{\mathbb{Z}}[n]=\left\{(a,b)\mid a\geqslant n\;\&\; b\geqslant n\right\}
\end{equation*}
is an inverse subsemigroup of $\mathscr{C}_{\mathbb{Z}}$ which  is isomorphic to the bicyclic semigroup ${\mathscr{C}}(p,q)$ by the map
\begin{equation*}
h\colon\mathscr{C}_{\mathbb{Z}}[n]\rightarrow {\mathscr{C}}(p,q), \qquad (a,b)\mapsto q^{a-n}p^{b-n}.
\end{equation*}
\end{proposition}

Proposition~\ref{proposition-2.1} implies the following

\begin{corollary}\label{corollary-2.2}
For every integer $n$ the set $\mathscr{C}_{\mathbb{Z}}^{\mathbf{0}}[n]=\mathscr{C}_{\mathbb{Z}}[n]\bigsqcup\{\mathbf{0}\}$
is an inverse subsemigroup of $\mathscr{C}_{\mathbb{Z}}^{\mathbf{0}}$ which  is isomorphic to the bicyclic monoid ${\mathscr{C}}^{\mathbf{0}}$ with adjoined zero by the map $h\colon\mathscr{C}_{\mathbb{Z}}^{\mathbf{0}}[n]\rightarrow {\mathscr{C}}^{\mathbf{0}}$, $(a,b)\mapsto q^{a-n}p^{b-n}$ and ${\mathbf{0}}\mapsto {\mathbf{0}}$.
\end{corollary}

\begin{lemma}\label{lemma-2.3}
Let $\tau$ be a non-discrete Hausdorff shift-continuous topology on $\mathscr{C}_{\mathbb{Z}}^{\mathbf{0}}$. Then  $\mathscr{C}_{\mathbb{Z}}^{\mathbf{0}}[n]$ is a non-discrete subsemigroup of $\left(\mathscr{C}_{\mathbb{Z}}^{\mathbf{0}},\tau\right)$ for any integer $n$.
\end{lemma}

\begin{proof}
First we observe that by Theorem~1 from \cite{Fihel-Gutik-2011} all non-zero elements of the semigroup $\mathscr{C}_{\mathbb{Z}}^{\mathbf{0}}$ are isolated points in $\left(\mathscr{C}_{\mathbb{Z}}^{\mathbf{0}},\tau\right)$.

Suppose to the contrary that there exist a non-discrete Hausdorff shift-continuous topology $\tau$ on $\mathscr{C}_{\mathbb{Z}}^{\mathbf{0}}$ and an integer $n$ such that $\mathscr{C}_{\mathbb{Z}}^{\mathbf{0}}[n]$ is a discrete subsemigroup of $\left(\mathscr{C}_{\mathbb{Z}}^{\mathbf{0}},\tau\right)$. Fix an arbitrary open neighbourhood $U(\mathbf{0})$ of zero $\mathbf{0}$ in $\left(\mathscr{C}_{\mathbb{Z}}^{\mathbf{0}},\tau\right)$ such that $U(\mathbf{0})\cap \mathscr{C}_{\mathbb{Z}}^{\mathbf{0}}[n]=\{\mathbf{0}\}$. Then the separate continuity of the semigroup operation in $\left(\mathscr{C}_{\mathbb{Z}}^{\mathbf{0}},\tau\right)$ implies that there exists an open neighbourhood $V(\mathbf{0})\subseteq U(\mathbf{0})$ of zero $\mathbf{0}$ in $\left(\mathscr{C}_{\mathbb{Z}}^{\mathbf{0}},\tau\right)$ such  that $(n,n)\cdot V(\mathbf{0})\cdot(n,n)\subseteq U(\mathbf{0})$. Our assumption implies that every open neighbourhood $W(\mathbf{0})\subseteq U(\mathbf{0})$ of zero $\mathbf{0}$ in $\left(\mathscr{C}_{\mathbb{Z}}^{\mathbf{0}},\tau\right)$ contains infinitely many points $(x,y)$ with the property: $x\leqslant n$ or $y\leqslant n$. Then for any non-zero $(x,y)\in V(\mathbf{0})$ by formula \eqref{eq-0.1} we have that
\begin{equation*}
  (n,n)\cdot (x,y)\cdot(n,n) =(n,n-x+y)\cdot(n,n)=
\left\{
  \begin{array}{ll}
    (n+x-y,n), & \hbox{if~} y\leqslant x; \\
    (n,n-x+y), & \hbox{if~} y\geqslant x,
  \end{array}
\right.
\end{equation*}
and hence
$
(n,n)\cdot V(\mathbf{0})\cdot(n,n)\cap \mathscr{C}_{\mathbb{Z}}[n]\neq\varnothing
$
which contradicts the assumption $U(\mathbf{0})\cap \mathscr{C}_{\mathbb{Z}}^{\mathbf{0}}[n]=\{\mathbf{0}\}$. The obtained contradiction implies the statement of the lemma.
\end{proof}

For an arbitrary non-zero element $(a,b)\in \mathscr{C}_{\mathbb{Z}}^{\mathbf{0}}$ we denote
\begin{equation*}
  {\uparrow}_{\preccurlyeq}(a,b)=\left\{(x,y)\in\mathscr{C}_{\mathbb{Z}}\colon (a,b)\preccurlyeq(x,y) \right\},
\end{equation*}
where $\preccurlyeq$ is the natural partial order on $\mathscr{C}_{\mathbb{Z}}^{\mathbf{0}}$. It is obvious that
\begin{equation*}
  {\uparrow}_{\preccurlyeq}(a,b)=\left\{(x,y)\in\mathscr{C}_{\mathbb{Z}} \colon a-b=x-y \hbox{~and~} x\leqslant a \hbox{~in~} (\mathbb{Z},\leqslant) \right\}.
\end{equation*}

\begin{lemma}\label{lemma-2.4}
Let $(a,b),(c,d),(e,f)\in \mathscr{C}_{\mathbb{Z}}$ be such that $(a,b)\cdot(c,d)=(e,f)$. Then the following statements hold.
\begin{itemize}
  \item[$(i)$] If $b\leqslant c$ then $(x,y)\cdot(c,d)=(e,f)$ for any $(x,y)\in{\uparrow}_{\preccurlyeq}(a,b)$, and moreover there exists a minimal element $(\hat{a},\hat{b})\preccurlyeq(a,b)$ in $\mathscr{C}_{\mathbb{Z}}$ such that $(\hat{a},\hat{b})\cdot(c,d)=(e,f)$. Also, there exist no other elements $(x,y)\in\mathscr{C}_{\mathbb{Z}}$ with the property $(x,y)\cdot(c,d)=(e,f)$.

  \item[$(ii)$] If $b\geqslant c$ then $(a,b)\cdot(x,y)=(e,f)$ for any $(x,y)\in{\uparrow}_{\preccurlyeq}(c,d)$, and moreover there exists a minimal element $(\hat{c},\hat{d})\preccurlyeq(c,d)$ in $\mathscr{C}_{\mathbb{Z}}$ such that $(a,b)\cdot(\hat{c},\hat{d})=(e,f)$. Also, there exist no other elements $(x,y)\in\mathscr{C}_{\mathbb{Z}}$ with the property $(a,b)\cdot(x,y)=(e,f)$.
\end{itemize}
\end{lemma}

\begin{proof}
$(i)$ Since $b\leqslant c$ the semigroup operation of $\mathscr{C}_{\mathbb{Z}}$ implies that $(b,b)\cdot(c,d)=(c,d)$. Also, if $(a,b)\preccurlyeq(x,y)$ then Lemma~1.4.6(5) from \cite{Lawson-1998} implies that
$$
(x,y)\cdot(b,b)=(x,y)\cdot(a,b)^{-1}\cdot(a,b)=(a,b),
$$
and hence we have that
\begin{equation*}
  (x,y)\cdot(c,d)=(x,y)\cdot((b,b)\cdot(c,d))=((x,y)\cdot(b,b))\cdot(c,d)=(a,b)\cdot(c,d)=(e,f).
\end{equation*}

We put $(\hat{a},\hat{b})=(a-b+c,c)$. Then $(\hat{a},\hat{b})\preccurlyeq(a,b)$ and formula \eqref{eq-0.1} implies that the element $(\hat{a},\hat{b})$ is required.

The last statement follows from Proposition~2,1 from \cite{Fihel-Gutik-2011} and formula \eqref{eq-0.1}.

\smallskip

The proof of statement $(ii)$ is similar.
\end{proof}

\begin{lemma}\label{lemma-2.5}
Let $\tau$ be a non-discrete Hausdorff shift-continuous topology on $\mathscr{C}_{\mathbb{Z}}^{\mathbf{0}}$. Then the natural partial order $\preccurlyeq$ is closed on $\left(\mathscr{C}_{\mathbb{Z}}^{\mathbf{0}},\tau\right)$ and ${\uparrow}_{\preccurlyeq}(a,b)$ is an open-and-closed subset of $\left(\mathscr{C}_{\mathbb{Z}}^{\mathbf{0}},\tau\right)$ for any non-zero element $(a,b)$ of $\mathscr{C}_{\mathbb{Z}}^{\mathbf{0}}$.
\end{lemma}

\begin{proof}
By Theorem~1 of \cite{Fihel-Gutik-2011} all non-zero elements of the semigroup $\mathscr{C}_{\mathbb{Z}}^{\mathbf{0}}$ are isolated points in $\left(\mathscr{C}_{\mathbb{Z}}^{\mathbf{0}},\tau\right)$. Since $\mathbf{0}\preccurlyeq(a,b)$ for any $(a,b)\in \mathscr{C}_{\mathbb{Z}}^{\mathbf{0}}$, the above implies the first statement of the lemma.

The definition of the natural partial order $\preccurlyeq$ on $\mathscr{C}_{\mathbb{Z}}^{\mathbf{0}}$ and the separate continuity of the semigroup operation on $\left(\mathscr{C}_{\mathbb{Z}}^{\mathbf{0}},\tau\right)$ imply the second statement, because
\begin{equation*}
{\uparrow}_{\preccurlyeq}(a,b)=\big\{(x,y)\in\mathscr{C}_{\mathbb{Z}}^{\mathbf{0}}\colon (a,a)\cdot (x,y)=(a,b)\big\}.
\end{equation*}
\end{proof}

\begin{proposition}\label{proposition-2.6}
Let the semigroup $\mathscr{C}_{\mathbb{Z}}^{\mathbf{0}}$ admits a non-discrete Hausdorff locally compact shift-continuous topology $\tau$. Then the following statements hold:
\begin{itemize}
  \item[$(i)$] for any open neighbourhood $U(\mathbf{0})$ of zero there exists a compact-and-open neighbourhood $V(\mathbf{0})\subseteq U(\mathbf{0})$ of $\mathbf{0}$ in $\left(\mathscr{C}_{\mathbb{Z}}^{\mathbf{0}},\tau\right)$;
  \item[$(ii)$] the set ${\uparrow}_{\preccurlyeq}(a,b)\cap U(\mathbf{0})$ is finite for any compact-and-open neighbourhood $V(\mathbf{0})\subseteq U(\mathbf{0})$ of the zero $\mathbf{0}$ in $\left(\mathscr{C}_{\mathbb{Z}}^{\mathbf{0}},\tau\right)$ and any non-zero element $(a,b)$ of $\mathscr{C}_{\mathbb{Z}}^{\mathbf{0}}$;
  \item[$(iii)$] for any open neighbourhood $U(\mathbf{0})$ of zero in $\left(\mathscr{C}_{\mathbb{Z}}^{\mathbf{0}},\tau\right)$ and any integer $n$ the set $U(\mathbf{0})\setminus \mathscr{C}_{\mathbb{Z}}^{\mathbf{0}}[n]$ is finite.
\end{itemize}
\end{proposition}

\begin{proof}
Statement $(i)$ follows from Theorem~1 of \cite{Fihel-Gutik-2011} and the local compactness of the space $\left(\mathscr{C}_{\mathbb{Z}}^{\mathbf{0}},\tau\right)$.

Statement $(ii)$ follows from Lemma~\ref{lemma-2.5} and Theorem~1 of \cite{Fihel-Gutik-2011}.

$(iii)$ It is obvious that $\mathscr{C}_{\mathbb{Z}}^{\mathbf{0}}[n]=(n,n)\cdot\mathscr{C}_{\mathbb{Z}}^{\mathbf{0}}\cdot(n,n)$ for any integer $n$. This implies that $\mathscr{C}_{\mathbb{Z}}^{\mathbf{0}}[n]$ is a closed subset of $\left(\mathscr{C}_{\mathbb{Z}}^{\mathbf{0}},\tau\right)$ because $\mathscr{C}_{\mathbb{Z}}^{\mathbf{0}}[n]$ is a retract of the space $\left(\mathscr{C}_{\mathbb{Z}}^{\mathbf{0}},\tau\right)$, and hence by Corollary~3.3.10 from \cite{Engelking-1989} it is locally compact.
Since the topology $\tau$ is non-discrete, Lemma~\ref{lemma-2.3} and Theorem~1 from \cite{Gutik-2015} imply that $\mathscr{C}_{\mathbb{Z}}^{\mathbf{0}}[n]$ is a compact subspace of $\left(\mathscr{C}_{\mathbb{Z}}^{\mathbf{0}},\tau\right)$. Finally, we apply Theorem~1 from \cite{Fihel-Gutik-2011}.
\end{proof}

Next we shall construct an example a non-discrete Hausdorff locally compact shift-continuous topology on the semigroup $\mathscr{C}_{\mathbb{Z}}^{\mathbf{0}}$ which is neither compact nor discrete.

\begin{example}\label{example-2.7}
Let $\left\{x_n\right\}_{n\in\mathbb{N}}$ and $\left\{y_n\right\}_{n\in\mathbb{N}}$ be two increasing sequences of positive integers with the following properties: $x_1, y_1>1$ and
\begin{equation*}
  x_n+1<x_{n+1} \qquad \hbox{and} \qquad 2<y_n+1<y_{n+1}, \qquad \hbox{~for any~} n\in\mathbb{N}.
\end{equation*}
We denote
\begin{equation*}
A_0={\uparrow}_{\preccurlyeq}(0,0)\cup \bigcup_{i=1}^{x_1-1}{\uparrow}_{\preccurlyeq}(0,-i)\cup \bigcup_{j=1}^{y_1-1}{\uparrow}_{\preccurlyeq}(-j,0)
\end{equation*}
and
\begin{align*}
  A_n^d=\bigcup_{i=x_n}^{x_{n+1}-1}{\uparrow}_{\preccurlyeq}(-x_n,-i);  \qquad
  A_n^l=\bigcup_{j=y_n}^{y_{n+1}-1}{\uparrow}_{\preccurlyeq}(-j,-y_n), \qquad \hbox{for any positive integer} \; n.
\end{align*}

Next, we put
$$
  D=A_0\cup \displaystyle\bigcup_{i\in\mathbb{N}} \left(A_i^d\cup A_i^l\right).
$$
For finitely many $(a_1,b_1),\ldots,(a_k,b_k)\in \mathscr{C}_{\mathbb{Z}}$ we denote
\begin{equation*}
  U_{(a_1,b_1),\ldots,(a_k,b_k)}=\mathscr{C}_{\mathbb{Z}}^{\mathbf{0}}\setminus\left(D\cup{\uparrow}_{\preccurlyeq}(a_1,b_1)\cup\cdots\cup {\uparrow}_{\preccurlyeq}(a_k,b_k) \right).
\end{equation*}

We define a topology $\tau_{\footnotesize{\{x_n\}}}^{\footnotesize{\{y_n\}}}$ on the semigroup $\mathscr{C}_{\mathbb{Z}}^{\mathbf{0}}$ in the following way:
\begin{itemize}
  \item[$(i)$] all non-zero elements of $\mathscr{C}_{\mathbb{Z}}^{\mathbf{0}}$ are isolated points;
  \item[$(ii)$] the family
  $
    \mathscr{B}_{\tau_{\footnotesize{\{x_n\}}}^{\footnotesize{\{y_n\}}}}^{\mathbf{0}}=\left\{ U_{(a_1,b_1),\ldots,(a_k,b_k)}\colon (a_1,b_1),\ldots,(a_k,b_k)\in \mathscr{C}_{\mathbb{Z}}, \; k\in\mathbb{N}\right\}
  $
  is the base of the topology $\tau_{\footnotesize{\{x_n\}}}^{\footnotesize{\{y_n\}}}$ at zero $\mathbf{0}$.
\end{itemize}
\end{example}

\begin{proposition}\label{proposition-2.8}
\begin{itemize}
  \item[$(1)$] The set ${\uparrow}_{\preccurlyeq}(a,b)\setminus D$ is finite for any $(a,b)\in \mathscr{C}_{\mathbb{Z}}$.
  \item[$(2)$] $D$ is a compact subset of the space $\Big(\mathscr{C}_{\mathbb{Z}}^{\mathbf{0}},\tau_{\footnotesize{\{x_n\}}}^{\footnotesize{\{y_n\}}}\Big)$.
  \item[$(3)$] The space $\Big(\mathscr{C}_{\mathbb{Z}}^{\mathbf{0}},\tau_{\footnotesize{\{x_n\}}}^{\footnotesize{\{y_n\}}}\Big)$ is locally compact and Hausdorff.
\end{itemize}
\end{proposition}

\begin{proof}
$(1)$ The statement is trivial in the case when $(a,b)\in D$, and hence we assume that $(a,b)\notin D$. Thus, we consider the followwing cases.
\begin{itemize}
  \item[$(i)$] If $a=b$ then ${\uparrow}_{\preccurlyeq}(a,b)\setminus D=\left\{(1,1),\ldots,(a,a)\right\}$.

  \item[$(ii)$] Suppose that $a<b$. Then either there exists a positive integer $i\geqslant 1$ such that $y_{i}\leqslant b-a<y_{i+1}$ or $b-a<y_1$. In the first case we have that
      \begin{equation*}
      {\uparrow}_{\preccurlyeq}(a,b)\setminus D=\left\{(-i+1-b+a,-i+1),\ldots,(a,b)\right\}=\bigcup\left\{(k-b+a,k)\colon k=-i+1,\ldots,b\right\}.
      \end{equation*}
      In the second case  we have that $b>0$ and hence
      \begin{equation*}
      {\uparrow}_{\preccurlyeq}(a,b)\setminus D=\left\{(1-b+a,1),\ldots,(a,b)\right\}=\bigcup\left\{(k-b+a,k)\colon k=1,\ldots,b\right\}.
      \end{equation*}

  \item[$(iii)$] Suppose that $a>b$. Then either there exists a positive integer $j\geqslant 1$ such that $x_{j}\leqslant a-b<x_{j+1}$ or $a-b<x_1$. In the first case we have that
      \begin{equation*}
      {\uparrow}_{\preccurlyeq}(a,b)\setminus D=\left\{(-j+1,-j+1-a+b),\ldots,(a,b)\right\}=\bigcup\left\{(k,k-a+b)\colon k=-j+1,\ldots,a\right\}.
      \end{equation*}
      In the second case we have that $a>0$ and hence
      \begin{equation*}
      {\uparrow}_{\preccurlyeq}(a,b)\setminus D=\left\{(1,1-a+b),\ldots,(a,b)\right\}=\bigcup\left\{(k,k-a+b)\colon k=1,\ldots,a\right\}.
      \end{equation*}
\end{itemize}

Statement $(2)$ follows from $(1)$.

Since all non-zero elements of $\mathscr{C}_{\mathbb{Z}}^{\mathbf{0}}$ are isolated points in $\Big(\mathscr{C}_{\mathbb{Z}}^{\mathbf{0}},\tau_{\footnotesize{\{x_n\}}}^{\footnotesize{\{y_n\}}}\Big)$ statement $(3)$ follows from $(2)$.
\end{proof}

For any non-zero element $(a,b)$ of $\mathscr{C}_{\mathbb{Z}}^{\mathbf{0}}$ we denote
\begin{equation*}
  S^{\underline{b}_{\uparrow}}=\left\{(x,y)\in\mathscr{C}_{\mathbb{Z}}\colon y\geqslant b \right\}\cup \left\{\mathbf{0}\right\} \qquad \hbox{and} \qquad  S^{\overrightarrow{{\shortmid}a}}=\left\{(x,y)\in\mathscr{C}_{\mathbb{Z}}\colon x\geqslant a \right\}\cup \left\{\mathbf{0}\right\}.
\end{equation*}

It is obvious that $(a,b)\mathscr{C}_{\mathbb{Z}}^{\mathbf{0}}=S^{\overrightarrow{{\shortmid}a}}$  and $\mathscr{C}_{\mathbb{Z}}^{\mathbf{0}}(a,b)=S^{\underline{b}_{\uparrow}}$ for any non-zero $(a,b)\in \mathscr{C}_{\mathbb{Z}}^{\mathbf{0}}$.

\begin{theorem}\label{theorem-2.9}
$\Big(\mathscr{C}_{\mathbb{Z}}^{\mathbf{0}},\tau_{\footnotesize{\{x_n\}}}^{\footnotesize{\{y_n\}}}\Big)$ is a semitopological semigroup.
\end{theorem}

\begin{proof}
By the definition of the topology $\tau_{\footnotesize{\{x_n\}}}^{\footnotesize{\{y_n\}}}$ it is sufficient to prove
that the left and right shifts of $\mathscr{C}_{\mathbb{Z}}^{\mathbf{0}}$ are continuous at zero $\mathbf{0}$.  

Fix an arbitrary basic open neighbourhood $U_{(a_1,b_1),\ldots,(a_k,b_k)}$ of zero $\mathbf{0}$ in $\Big(\mathscr{C}_{\mathbb{Z}}^{\mathbf{0}},\tau_{\footnotesize{\{x_n\}}}^{\footnotesize{\{y_n\}}}\Big)$.

The definition of the topology $\tau_{\footnotesize{\{x_n\}}}^{\footnotesize{\{y_n\}}}$ implies that  there exist finitely many non-zero elements \linebreak $(e_1,f_1),\ldots,(e_m,f_m)$ of the semigroup $\mathscr{C}_{\mathbb{Z}}^{\mathbf{0}}$ with $e_1,\ldots,e_m\geqslant a$ such that
\begin{equation*}
  U_{(a_1,b_1),\ldots,(a_k,b_k)}\cap S^{\overrightarrow{{\shortmid}a}}=S^{\overrightarrow{{\shortmid}a}}\setminus \left({\uparrow}_{\preccurlyeq}(e_1,f_1)\cup\cdots\cup{\uparrow}_{\preccurlyeq}(e_m,f_m)\right).
\end{equation*}

Since $(a,b)\mathscr{C}_{\mathbb{Z}}^{\mathbf{0}}=S^{\overrightarrow{{\shortmid}a}}$ by Lemma~\ref{lemma-2.4}$(ii)$ there exist minimal elements $(\hat{c}_1,\hat{d}_1),\ldots,(\hat{c}_m,\hat{d}_m)$ in $\mathscr{C}_{\mathbb{Z}}$ such that
$$
(a,b)\cdot(\hat{c}_1,\hat{d}_1)=(e_1,f_1),\; \ldots,\;(a,b)\cdot(\hat{c}_m,\hat{d}_m)=(e_m,f_m).
$$

Then the last equalities imply that
$$
(a,b)\cdot U_{(\hat{c}_1,\hat{d}_1),\ldots,(\hat{c}_m,\hat{d}_m)}\subseteq U_{(a_1,b_1),\ldots,(a_k,b_k)}.
$$

Similarly, there exist finitely many non-zero elements $(e_1,f_1),\ldots,(e_p,f_p)$ of the semigroup $\mathscr{C}_{\mathbb{Z}}^{\mathbf{0}}$ with $f_1,\ldots,f_p\geqslant b$ such that
$$
  U_{(a_1,b_1),\ldots,(a_k,b_k)}\cap S^{\underline{b}_{\uparrow}}=S^{\underline{b}_{\uparrow}}\setminus \left({\uparrow}_{\preccurlyeq}(e_1,f_1)\cup\cdots\cup{\uparrow}_{\preccurlyeq}(e_p,f_p)\right).
$$
Since $\mathscr{C}_{\mathbb{Z}}^{\mathbf{0}}(a,b)=S^{\underline{b}_{\uparrow}}$, by Lemma~\ref{lemma-2.4}$(i)$ there exist minimal elements $(\hat{c}_1,\hat{d}_1),\ldots,(\hat{c}_p,\hat{d}_p)$ in $\mathscr{C}_{\mathbb{Z}}$ such that
$$
(\hat{c}_1,\hat{d}_1)\cdot(a,b)=(e_1,f_1),\; \ldots,\;(\hat{c}_p,\hat{d}_p)\cdot(a,b)=(e_p,f_p).
$$
Then the last equalities imply that
$$
U_{(\hat{c}_1,\hat{d}_1),\ldots,(\hat{c}_p,\hat{d}_p)}\cdot(a,b)\subseteq U_{(a_1,b_1),\ldots,(a_k,b_k)},
$$
which completes the proof of the separate continuity of the semigroup operation in $\Big(\mathscr{C}_{\mathbb{Z}}^{\mathbf{0}},\tau_{\footnotesize{\{x_n\}}}^{\footnotesize{\{y_n\}}}\Big)$.
\end{proof}

If in Example~\ref{example-2.7} we put $x_i=y_i$ for any $i\in\mathbb{N}$ and denote $\tau_{\footnotesize{\{x_n\}}}=\tau_{\footnotesize{\{x_n\}}}^{\footnotesize{\{y_n\}}}$ then
\begin{equation*}
\left(U_{(a_1,b_1),\ldots,(a_k,b_k)}\right)^{-1}=U_{(b_1,a_1),\ldots,(b_k,a_k)}
\end{equation*}
for any $a_1,b_1,\ldots,a_k,b_k\in {\mathbb{Z}}$. This and Theorem~\ref{theorem-2.9} imply the following corollary:

\begin{corollary}
$\Big(\mathscr{C}_{\mathbb{Z}}^{\mathbf{0}},\tau_{\footnotesize{\{x_n\}}}\Big)$ is a Hausdorff locally compact semitopological semigroup with continuous inversion.
\end{corollary}

Theorem~\ref{theorem-2.9} implies that on the semigroup $\mathscr{C}_{\mathbb{Z}}^{\mathbf{0}}$ there exist $\mathfrak{c}$ many Hausdorff locally compact shift-continuous topologies. But Lemma~\ref{lemma-2.3} implies the following counterpart of Corollary~1 from \cite{Gutik-2015}:

\begin{corollary}
Every Hausdorff locally compact semigroup topology on the semigroup $\mathscr{C}_{\mathbb{Z}}^{\mathbf{0}}$ is discrete.
\end{corollary}

\section{Minimal shift-continuous and inverse semigroup topologies on $\mathscr{C}_{\mathbb{Z}}^{\mathbf{0}}$}

The concept of minimal topological groups was introduced independently in the early 1970’s by Do\^{\i}tchinov \cite{Doitchinov-1972} and Stephenson \cite{Stephenson-Jr-1971}. Both authors were motivated by the theory of minimal topological spaces, which was well understood at that time (cf. \cite{Berri-Porter-Stephenson-Jr-1971}). More than 20 years earlier L. Nachbin \cite{Nachbin-1949} had studied minimality in the context of division rings, and B. Banaschewski \cite{Banaschewski-1974} investigated minimality in the more general setting of topological algebras.
The concept of minimal topological semigroups was introduced in \cite{Gutik-Pavlyk-2005}.

\begin{definition}[\cite{Gutik-Pavlyk-2005}]\label{definition-3.1}
A Hausdorff semitopological (resp., topological, topological inverse) semigroup $(S,\tau)$ is said to be \emph{minimal} if no Hausdorff shift-continuous (resp., semigroup, semigroup inverse) topology on $S$ is strictly contained in $\tau$. If $(S,\tau)$ is minimal semitopological (resp., topological, topological inverse) semigroup, then $\tau$ is called \emph{minimal} \emph{shift-continuous} (resp., \emph{semigroup}, \emph{semigroup inverse}) topology.
\end{definition}

It is obvious that every Hausdorff compact shift-continuous (resp., semigroup, semigroup inverse) topology on a semigroup $S$ is a minimal shift-continuous (resp., semigroup, semigroup inverse) topology on $S$. But an infinite semigroup of units admits a unique compact shift-continuous topology and non-compact minimal semigroup and semigroup inverse  topologies \cite{Gutik-Pavlyk-2005}. Similar results were obtained in \cite{Bardyla-Gutik-2019} for the bicyclic monoid with adjoined zero $\mathscr{C}^0$.

\begin{example}\label{example-3.2}
For finitely many $(a_1,b_1),\ldots,(a_k,b_k)\in \mathscr{C}_{\mathbb{Z}}$ we denote
\begin{equation*}
  U_{(a_1,b_1),\ldots,(a_k,b_k)}^{\uparrow}=\mathscr{C}_{\mathbb{Z}}^{\mathbf{0}}\setminus\left({\uparrow}_{\preccurlyeq}(a_1,b_1)\cup\cdots\cup {\uparrow}_{\preccurlyeq}(a_k,b_k) \right).
\end{equation*}

We define a topology $\tau_{\textsf{min}}^{\textsf{sh}}$ on the semigroup $\mathscr{C}_{\mathbb{Z}}^{\mathbf{0}}$ in the following way:
\begin{itemize}
  \item[$(i)$] all non-zero elements of $\mathscr{C}_{\mathbb{Z}}^{\mathbf{0}}$ are isolated points;
  \item[$(ii)$] the family
 $
    \mathscr{B}_{\tau_{\textsf{min}}^{\textsf{sh}}}^{\mathbf{0}}=\left\{ U_{(a_1,b_1),\ldots,(a_k,b_k)}^{\uparrow}\colon (a_1,b_1),\ldots,(a_k,b_k)\in \mathscr{C}_{\mathbb{Z}}, \; k\in\mathbb{N}\right\}
  $
  is the base of the topology $\tau_{\textsf{min}}^{\textsf{sh}}$ at zero $\mathbf{0}$.
\end{itemize}

We observe that by Lemma~\ref{lemma-2.5} the space $\left(\mathscr{C}_{\mathbb{Z}}^{\mathbf{0}},\tau_{\textsf{min}}^{\textsf{sh}}\right)$ is Hausdorff, $0$-dimensional and scattered, and hence it is regular. Since the base $\mathscr{B}_{\tau_{\textsf{min}}^{\textsf{sh}}}^{\mathbf{0}}$ is countable, by the Urysohn Metrization Theorem (see \cite[p.~123, Theorem~16]{Kelley-1975}) the space  $\left(\mathscr{C}_{\mathbb{Z}}^{\mathbf{0}},\tau_{\textsf{min}}^{\textsf{sh}}\right)$ is metrizable and hence by Corollary~4.1.13 of \cite{Engelking-1989} it is perfectly normal.
\end{example}

\begin{proposition}\label{proposition-3.3}
$\left(\mathscr{C}_{\mathbb{Z}}^{\mathbf{0}},\tau_{\textsf{min}}^{\textsf{sh}}\right)$ is a minimal semitopological semigroup with continuous inversion.
\end{proposition}

\begin{proof}
The definition of the topology $\tau_{\textsf{min}}^{\textsf{sh}}$ implies that it is sufficient to prove that the left and right shifts of $\mathscr{C}_{\mathbb{Z}}^{\mathbf{0}}$ are continuous at zero $\mathbf{0}$. 

Fix any non-zero element $(a, b) \in\mathscr{C}_{\mathbb{Z}}^{\mathbf{0}}$ and any basic open neighbourhood $U_{(a_1,b_1),\ldots,(a_k,b_k)}^{\uparrow}$ of zero $\mathbf{0}$ in $\left(\mathscr{C}_{\mathbb{Z}}^{\mathbf{0}},\tau_{\textsf{min}}^{\textsf{sh}}\right)$.

The definition of the topology $\tau_{\textsf{min}}^{\textsf{sh}}$ implies that  there exist finitely many non-zero elements \linebreak $(e_1,f_1),\ldots,(e_m,f_m)$ of the semigroup $\mathscr{C}_{\mathbb{Z}}^{\mathbf{0}}$ with $e_1,\ldots,e_m\geqslant a$ such that
\begin{equation*}
  U_{(a_1,b_1),\ldots,(a_k,b_k)}^{\uparrow}\cap S^{\overrightarrow{{\shortmid}a}}=S^{\overrightarrow{{\shortmid}a}}\setminus \left({\uparrow}_{\preccurlyeq}(e_1,f_1)\cup\cdots\cup{\uparrow}_{\preccurlyeq}(e_m,f_m)\right).
\end{equation*}
Since $(a,b)\mathscr{C}_{\mathbb{Z}}^{\mathbf{0}}=S^{\overrightarrow{{\shortmid}a}}$, by Lemma~\ref{lemma-2.4}$(ii)$ there exist minimal elements $(\hat{c}_1,\hat{d}_1),\ldots,(\hat{c}_m,\hat{d}_m)$ in $\mathscr{C}_{\mathbb{Z}}$ such that
$$
(a,b)\cdot(\hat{c}_1,\hat{d}_1)=(e_1,f_1),\; \ldots,\;(a,b)\cdot(\hat{c}_m,\hat{d}_m)=(e_m,f_m).
$$
Then the last equalities imply that
$$
(a,b)\cdot U_{(\hat{c}_1,\hat{d}_1),\ldots,(\hat{c}_m,\hat{d}_m)}^{\uparrow}\subseteq U_{(a_1,b_1),\ldots,(a_k,b_k)}^{\uparrow}.
$$

Again, by similar way there exists finitely many non-zero elements $(e_1,f_1),\ldots,(e_p,f_p)$ of the semigroup $\mathscr{C}_{\mathbb{Z}}^{\mathbf{0}}$ with $f_1,\ldots,f_p\geqslant b$ such that
$$
  U_{(a_1,b_1),\ldots,(a_k,b_k)}^{\uparrow}\cap S^{\underline{b}_{\uparrow}}=S^{\underline{b}_{\uparrow}}\setminus \left({\uparrow}_{\preccurlyeq}(e_1,f_1)\cup\cdots\cup{\uparrow}_{\preccurlyeq}(e_p,f_p)\right).
$$
Since $\mathscr{C}_{\mathbb{Z}}^{\mathbf{0}}(a,b)=S^{\underline{b}_{\uparrow}}$, Lemma~\ref{lemma-2.4}$(i)$ implies that there exist minimal elements $(\hat{c}_1,\hat{d}_1),\ldots,(\hat{c}_p,\hat{d}_p)$ in $\mathscr{C}_{\mathbb{Z}}$ such that
$$
(\hat{c}_1,\hat{d}_1)\cdot(a,b)=(e_1,f_1),\; \ldots,\;(\hat{c}_p,\hat{d}_p)\cdot(a,b)=(e_p,f_p).
$$
Then the last equalities imply that
$$
U_{(\hat{c}_1,\hat{d}_1),\ldots,(\hat{c}_p,\hat{d}_p)}^{\uparrow}\cdot(a,b)\subseteq U_{(a_1,b_1),\ldots,(a_k,b_k)}^{\uparrow},
$$
which completes the proof of the separate continuity of the semigroup operation in $\left(\mathscr{C}_{\mathbb{Z}}^{\mathbf{0}},\tau_{\textsf{min}}^{\textsf{sh}}\right)$.
Also, since
$$
\left(U_{(a_1,b_1),\ldots,(a_k,b_k)}^{\uparrow}\right)^{-1}=U_{(b_1,a_1),\ldots,(b_k,a_k)}^{\uparrow},
$$
for any $(a_1,b_1),\ldots,(a_k,b_k)\in \mathscr{C}_{\mathbb{Z}}$, the inversion is continuous in $\left(\mathscr{C}_{\mathbb{Z}}^{\mathbf{0}},\tau_{\textsf{min}}^{\textsf{sh}}\right)$ as well.

Lemma~\ref{lemma-2.5} implies that $\tau_{\textsf{min}}^{\textsf{sh}}$ is the coarsest Hausdorff shift-continuous topology on $\mathscr{C}_{\mathbb{Z}}^{\mathbf{0}}$, and hence $\left(\mathscr{C}_{\mathbb{Z}}^{\mathbf{0}},\tau_{\textsf{min}}^{\textsf{sh}}\right)$ is a minimal semitopological semigroup.
\end{proof}

\begin{example}\label{example-3.4}
We define the topology $\tau_{\textsf{min}}^{\textsf{i}}$ on the semigroup $\mathscr{C}_{\mathbb{Z}}^{\mathbf{0}}$ in the following way:
\begin{itemize}
  \item[$(i)$] all non-zero elements of $\mathscr{C}_{\mathbb{Z}}^{\mathbf{0}}$ are isolated points in the topological space $\left(\mathscr{C}_{\mathbb{Z}}^{\mathbf{0}},\tau_{\textsf{min}}^{\textsf{i}}\right)$;
  \item[$(ii)$] the family $\mathscr{B}_{\tau_{\textsf{min}}^{\textsf{i}}}^{\mathbf{0}}=\left\{ S^{\overrightarrow{{\shortmid}a}}\cap S^{\underline{b}_{\uparrow}}\colon a,b\in \mathbb{Z}\right\}$ is the base of the topology $\tau_{\textsf{min}}^{\textsf{i}}$ at zero $\mathbf{0}$.
\end{itemize}

It is obvious that the space  $\left(\mathscr{C}_{\mathbb{Z}}^{\mathbf{0}},\tau_{\textsf{min}}^{\textsf{i}}\right)$ is Hausdorff, $0$-dimensional and scattered, and hence it is regular. Since the base $\mathscr{B}_{\tau_{\textsf{min}}^{\textsf{i}}}^{\mathbf{0}}$ is countable, similar as in Example~\ref{example-3.2} we get that the space  $\left(\mathscr{C}_{\mathbb{Z}}^{\mathbf{0}},\tau_{\textsf{min}}^{\textsf{i}}\right)$ is metrizable.
\end{example}

\begin{proposition}\label{proposition-3.6}
$\left(\mathscr{C}_{\mathbb{Z}}^{\mathbf{0}},\tau_{\textsf{min}}^{\textsf{i}}\right)$ is a minimal topological inverse semigroup.
\end{proposition}

\begin{proof}
We have that for any $a,b\in\mathbb{Z}$ and any non-zero element $(x,y)\in\mathscr{C}_{\mathbb{Z}}^{\mathbf{0}}$ there exists an integer $n$ such that $(x,y)\in\mathscr{C}_{\mathbb{Z}}^{\mathbf{0}}[n]$ and $S^{\overrightarrow{{\shortmid}a}}\cap S^{\underline{b}_{\uparrow}}\subseteq \mathscr{C}_{\mathbb{Z}}^{\mathbf{0}}[n]$. By Corollary~\ref{corollary-2.2} the semigroup $\mathscr{C}_{\mathbb{Z}}^{\mathbf{0}}[n]$ is isomorphic to the bicyclic monoid with adjoined zero $\mathscr{C}^\mathbf{0}$. Also, it is obviously that the topology $\tau_{\textsf{min}}^{\textsf{i}}$ induces the topology $\tau$ on $\mathscr{C}_{\mathbb{Z}}^{\mathbf{0}}[n]$ such that $\tau$ generates by the map $h\colon\mathscr{C}_{\mathbb{Z}}^{\mathbf{0}}[n]\rightarrow {\mathscr{C}}^{\mathbf{0}}$, $(a,b)\mapsto q^{a-n}p^{b-n}$ and ${\mathbf{0}}\mapsto {\mathbf{0}}$ the topology $\tau_{\textsf{min}}$ on $\mathscr{C}^\mathbf{0}$ from \cite{Bardyla-Gutik-2019}. Then the proof of Lemma~2 from \cite{Gutik-1996} implies that $\left(\mathscr{C}^\mathbf{0},\tau_{\textsf{min}}\right)$ is a Hausdorff topological semigroup. This and the above arguments imply that $\left(\mathscr{C}_{\mathbb{Z}}^{\mathbf{0}},\tau_{\textsf{min}}^{\textsf{i}}\right)$ is a topological inverse semigroup. The minimality of $\left(\mathscr{C}_{\mathbb{Z}}^{\mathbf{0}},\tau_{\textsf{min}}^{\textsf{i}}\right)$ as topological inverse semigroup follows from Lemma~\ref{lemma-2.5}, because
\begin{equation*}
  \mathscr{C}_{\mathbb{Z}}^{\mathbf{0}}\setminus (S^{\overrightarrow{{\shortmid}a}}\cap S^{\underline{b}_{\uparrow}})= \left\{(x,y)\colon (x,y){\cdot}(x,y)^{-1}\in {\uparrow}_{\preccurlyeq}(a{-}1,a{-}1)\right\}\cup \left\{(x,y)\colon (x,y)^{-1}{\cdot}(x,y)\in {\uparrow}_{\preccurlyeq}(b{-}1,b{-}1)\right\}.
\end{equation*}
\end{proof}

\section*{Acknowledgements}
We acknowledge Taras Banakh and the referee for useful important comments and suggestions.


\end{document}